\documentclass[letterpaper, 10 pt, conference]{ieeeconf}  % Comment this line out
                                                          \usepackage{mathrsfs}
\usepackage{amsmath}
\usepackage{amssymb}

\newtheorem{lemma}{Lemma}
\newtheorem{proposition}{Proposition}

\usepackage{mathrsfs}
\usepackage{amsmath}
\usepackage{enumerate}
\usepackage{marvosym}
\usepackage{algorithm}
\usepackage{algorithmic}

\newcommand{\tr}{\text{{Tr}}}
\newcommand{\col}{\text{{Col}}}
\newcommand{\vect}{\text{{vec}}}
\newcommand{\diag}{\text{{diag}}}

\def\onesvec{{\bf 1}}
\def\diag{\text{{diag}}}

\IEEEoverridecommandlockouts                              % This command is only
   \usepackage{graphicx}                                                       % use the \thanks command
\title{\LARGE \bf
Efficient Stabilization of Hybrid Coulomb Spacecraft Formations using Control Lyapunov Functions
}
\author{Adam M. Tahir % <-this % stops a space
\thanks{A. M. Tahir is with the William E. Boeing Department of Aeronautics and Astronautics, University of Washington, Seattle, WA, USA. Email: 
       {{\tt\small aerotahir@gmail.com}.}
}}

\begin{document}

\maketitle
\thispagestyle{empty}
\pagestyle{empty}

%%%%%%%%%%%%%%%%%%%%%%%%%%%%%%%%%%%%%%%%%%%%%%%%%%%%%%%%%%%%%%%%%%%%%%%%%%%%%%%%
\begin{abstract}
A control allocation algorithm using control Lyapunov functions to determine stabilizing charges and thrusts of hybrid Coulomb spacecraft formations (HCSFs) is presented. The goal is to stabilize a desired configuration while minimizing the thruster actuation and maximizing Coulomb actuation to minimize propellant usage. A proportion of the decrease of the control Lyapunov function is designated for Coulomb actuation and the rest is performed by thrusters. Simulations show that an 85\% reduction of propellant compared to using solely thrusters is attainable using the proposed algorithm. It is shown that the best role for thrusters in a HCSF is to provide small corrections that cannot be provided by Coulomb actuation.
\end{abstract}
\section{Introduction}
A hybrid Coulomb spacecraft formation (HCSF) is a formation of spacecraft that are equipped with two methods of actuation: active charge control systems and conventional thrusters. The active charge control systems change the charges of the spacecraft in the formation which changes the Coulomb forces acting between the spacecraft. This is appealing for formation control since changing the charge requires negligible amounts of propellant \cite{Prospects}. Coulomb forces are limited because they come in action-reaction pairs and act along the line-of-sight between spacecraft. Hence, the spacecraft in a HCSF also are equipped with conventional thrusters to provide forces that cannot be provided by Coulomb actuation alone.  A HCSF is over-actuated in the sense that all of the actuation can be performed using solely thrusters \cite{Allocation}. The challenge of HCSF control is to try to minimize the amount of thrust used and instead try to use Coulomb forces as much as possible. This minimizes the amount of propellant required to perform maneuvers. 

A major challenge of Coulomb spacecraft formation control is the nonlinearity of the dynamics with respect to the charge input. This is due to the fact that the Coulomb forces are proportional to the products of charges. For HCSFs of two spacecraft, the nonlinearity is easy to overcome.  In that case the input can be redefined as the product of the charges so that the dynamics become linear in the new input. This is why most HCSF results in the literature consider only formations of two spacecraft \cite{ArunHybrid,CoulombNonlin,Seo}. For formations of more than two spacecraft, the problem is much more difficult and a simple redefinition of inputs as the products of charges cannot be done without imposing additional constraints. For example, it is impossible for all charge products to be negative and the magnitudes of the charge products are interrelated. These constraints are difficult to impose computationally. 

Pettazzi {\it et al.} \cite{HybridCoulombIzzo} proposed an architecture for HCSF control with more than two spacecraft where a path-planning algorithm gives acceleration commands to a HCSF and the charges and thrusts are computed to match the acceleration commands using projection. The issue with this approach is that it is difficult to design a path-planning algorithm to give acceleration commands that are best suited to minimize thruster input. In this paper, the problem is approached in a different way where the charges and thrusts are computed directly to satisfy a stability condition. This is done using a control allocation approach.

In this paper, the control allocation algorithm proposed revolves around control Lyapunov functions (CLFs). Using CLFs as a method for control involves specifying a Lyapunov function and finding control inputs such that the derivative of the Lyapunov function is negative. For more information about CLFs see \cite[\S 9.7]{Khalil} and \cite[\S 10]{HybridControl} and the references therein. Control allocation using CLFs has been studied in the literature \cite{1429240,TJONNAS20051160}. These references approach the control allocation problem using a pointwise optimization approach where a cost function that enumerates the tradeoff between actuators is minimized subject to the condition that the CLF decreases. Rather than solve the optimization problem periodically, a continuous-time gradient descent method is employed. %CLFs have been used in the literature for optimal control \cite{Pointwise_min}, control allocation \, multi-agent coordination \cite{1068003}, bipedal robotics \cite{7079382}, safety-critical control  \cite{7782377}, and data-driven control \cite{10666531}. 

When the system dynamics are affine in the control, then the derivative of the CLF is affine in the input. Therefore, pointwise optimal control can be found by solving a quadratic program (QP) even if the underlying dynamics are nonlinear in the state \cite{ 7079382,7782377}.  In this paper, it is shown that the derivative of the CLF for HCSFs  is quadratic in the charge input and affine in the thrust input. Hence, the control allocation algorithm is derived by manipulating the quadratic expression of the derivative of the CLF. Pointwise optimal control allocation for HCSFs could be performed by solving a quadratically-constrained quadratic program (QCQP) that has one quadratic constraint. In this case, an optimal solution can be found in polynomial time even if the quadratic constraint is nonconvex \cite{park2017generalheuristicsnonconvexquadratically}. The proposed control allocation algorithm in this paper is inspired by the pointwise optimal allocation approach. The relationship of the proposed algorithm to the optimal solution of a nonconvex one constraint QCQP is discussed in this paper. 

%It is shown through simulation that a point-wise minimum norm approach to the input does not necessitate the best performance. 

%CLFs have been used for  

A related idea from spacecraft attitude control described in \cite{7963020} is autonomous blending of passive and active (ABPA) control. In the ABPA setup, when the Lyapunov function is decreasing sufficiently without input (e.g. through fortuitous external disturbances), then the input can be zero to reduce propellant consumption. This behavior is also a feature in the proposed control allocation algorithm in this paper.  

The contribution of this paper to the state-of-the-art is the proposition of an algorithm to compute charges and thrusts to stabilize desired configurations of HCSFs with more than two spacecraft. The proposed algorithm does not require an intermediate path-planning algorithm to provide acceleration commands. The stability of the proposed algorithm is discussed. In simulation, the algorithm is shown to provide an 85\% reduction in propellant compared to using solely thrusters to perform the same task for a four spacecraft formation. 

The rest of the paper is organized as follows: \S \ref{s:setup} sets up the problem, introduces notation, and presents some pertinent background information. \S \ref{s:algorithm} presents the proposed control allocation algorithm. \S \ref{s:qcqprel} discusses the relationship between the proposed algorithm and the optimal solution of a CLF-based QCQP. \S \ref{s:example} presents a numerical simulation and discusses the effect of the tradeoff parameter in the proposed algorithm. Finally, \S \ref{s:conc} concludes the paper. 

\subsection{Notation}
Let $\mathbb{R}$ denote the set of real numbers. The set of real $n$-dimensional vectors is denoted $\mathbb{R}^n$ and the set of real $n\times m$ matrices is denoted $\mathbb{R}^{n\times m}$.  %The positive real numbers and integers are denoted by $\mathbb{R}_+$ and $\mathbb{Z}_+$, respectively. 

The transpose is denoted using a superscript $\top$. The Hermitian of a matrix $M\in\mathbb{R}^{n\times n}$ is denoted $\text{He}(M)=\frac{1}{2}\left(M+M^\top\right)$. The trace of a matrix $M$ is denoted by $\tr(M)$. The set of symmetric $n\times n$ matrices is denoted by $\mathbb{S}^n$. A matrix $M\in\mathbb{S}^n$ is positive (resp. negative) semidefinite if its eigenvalues are nonnegative (resp. nonpositive), and $M\succeq 0$ (resp. $\preceq 0$) denotes that $M$ is positive (resp. negative) semidefinite. 

If $x_1, \dots, x_n$ are all vectors, then $\col(x_1, \dots, x_n)$ is a column vector that is produced by stacking $x_1$ through $x_n$ on top of each other.  If $M\in\mathbb{R}^{n\times m}$ is a matrix, then $\vect(M)$ is the column vector produced by stacking consecutive columns of $M$ on top of each other. A diagonal matrix $M\in\mathbb{R}^{n\times n}$ with diagonal elements $M_{11},\dots,M_{nn}\in\mathbb{R}$ is denoted by $M=\text{diag}(M_{11},\dots,M_{nn})$.

The $n\times n$ identity matrix is denoted by $I_n$, the vector of all ones of dimension $n$ is denoted by $\onesvec_n$. A matrix of all zeros of dimension $n\times m$ is denoted by $0_{n\times m}$. If the size of the identity matrix or a matrix of zeros is clear from context then they will be simply denoted by $I$ and $0$, respectively. 

The Kronecker product of matrix $A$ with matrix $B$ is denoted by $A\otimes B$. 

The partial derivative of a function $f(x)$ with respect to $x$ is denoted by $\frac{\partial f}{\partial x}$. The Lie derivative of a function $h$ along $f$ is $L_fh=\frac{\partial h}{\partial x}f$.

\section{Problem Setup}\label{s:setup}

\subsection{Formation Definition}
Consider a formation of $\mathcal{N}>1$ spacecraft embedded in $\mathbb{R}^d$, where $d=1$ for collinear formations, $d=2$ for coplanar formations, or $d=3$ for three-dimensional formations. The position and velocity relative to an inertial frame of the spacecraft indexed $i\in\{1,\dots,\mathcal{N}\}$ are denoted by $x_i\in\mathbb{R}^d$ and $v_i\in\mathbb{R}^d$, measured in m and m/s, respectively. It's mass is denoted by $m_i$ which is measured in kg. 

Each spacecraft is modeled as a point charge with charge $q_i\in\mathbb{R}$, measured in C, which can be adjusted by its active charge control system. Moreover, each spacecraft is equipped with a set of conventional thrusters that can produce thrust $T_i\in\mathbb{R}^d$, measured in N. The vector of all charges is $q=\col(q_1,\dots, q_\mathcal{N})\in\mathbb{R}^\mathcal{N}$, and the vector of all thrusts is $T=\col(T_1,\dots,T_\mathcal{N})\in\mathbb{R}^{d\mathcal{N}}$. Note that space weather effects and the effects of spacecraft geometry \cite{CoulombForce} are not considered. In Fig. \ref{f:CoulombSchematic}, a HCSF with $\mathcal{N}=3$ and $d=2$ is depicted.

%The vector of all positions in the formation is $x=\col(x_1,\dots, x_\mathcal{N})\in\mathbb{R}^{d\mathcal{N}}$, the vector of all charges is $q=\col(q_1,\dots, q_\mathcal{N})\in\mathbb{R}^\mathcal{N}$, and the vector of all thrusts is $T=\col(T_1,\dots,T_\mathcal{N})\in\mathbb{R}^{d\mathcal{N}}$. In Fig. \ref{f:CoulombSchematic}, an HCSF with $\mathcal{N}=3$ and $d=2$ is depicted. 
 \begin{figure}[!h]
\centerline{\includegraphics[width=0.8\columnwidth]{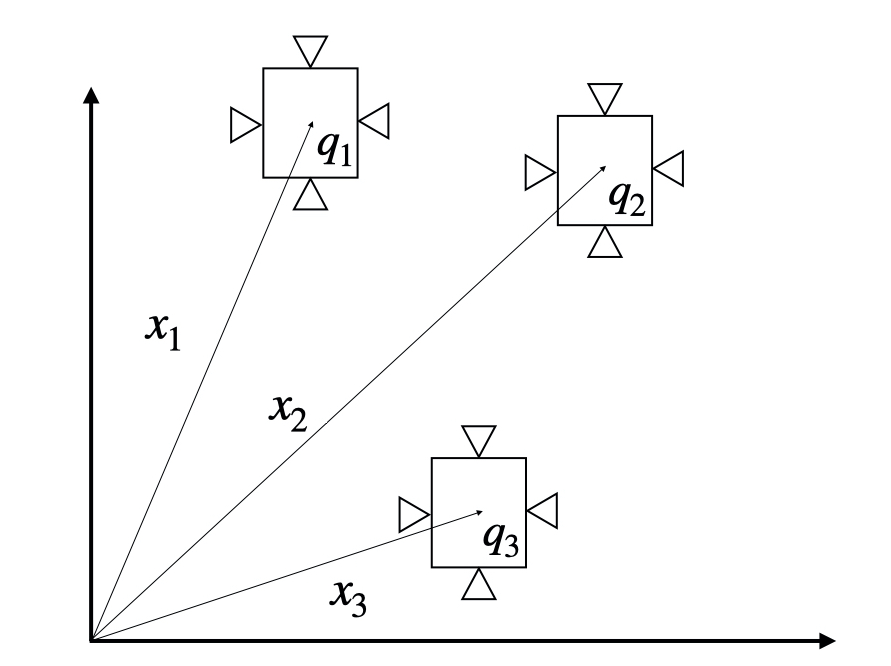}}
\caption{Diagram of a HCSF with $\mathcal{N}=3$ in $\mathbb{R}^2$. The positions of each spacecraft are $x_1,x_2,x_3\in\mathbb{R}^2$, and their respective charges are $q_1,q_2,q_3\in\mathbb{R}$. Each spacecraft is depicted with a set of thrusters that can generate thrusts $T_1,T_2,T_3\in\mathbb{R}^2$.}
\label{f:CoulombSchematic}
\end{figure}

The dynamics of a HCSF come from Newton's second law with the accelerations from Coulomb forces and thrust forces:
\begin{align}
\ddot x_i = \sum_{j=1 \atop j\ne i}^\mathcal{N}\frac{\kappa_c}{m_i}\frac{x_i-x_j}{\|x_i-x_j\|^3}q_iq_j+\frac{T_i}{m_i},\label{e:truedyn}
\end{align}
for all $i=1,\dots, \mathcal{N}$. Coulomb's constant is denoted $\kappa_c=8.99$e$+09$ Nm$^2$/C$^2$. 

\subsection{Relative Dynamics and Objective}
In this paper, the concern is the control of the spacecraft positions relative to each other. Therefore, the following relative coordinates are defined:
\begin{align}
\tilde{\xi}_i=x_{i+1}-x_1,\forall i=1,\dots,\mathcal{N}-1.\label{relativecoords}
\end{align}
Using \eqref{e:truedyn} and defining $\tilde{\xi}=\col(\tilde{\xi}_1,\dots,\tilde{\xi}_{\mathcal{N}-1}),$ the dynamics can be written compactly as follows:
\begin{align}
\ddot{\tilde{\xi}} = \tilde{g}_C(\tilde{\xi})\vect\left(qq^\top\right)+\tilde{g}_TT,\label{e:dyna}
\end{align} 
where $\tilde{g}_C:\mathbb{R}^{d(\mathcal{N}-1)}\to\mathbb{R}^{d(\mathcal{N}-1)\times \mathcal{N}^2}$ is a function which includes all of Coulomb's law's inverse-square terms and spacecraft masses. It is cumbersome to write out explicitly, but it can be derived using all of the terms in \eqref{e:truedyn} and the coordinate substitution \eqref{relativecoords}. The matrix $\tilde{g}_T\in\mathbb{R}^{d(\mathcal{N}-1)\times d\mathcal{N}}$ includes all the spacecraft masses from Newton's second law:
\begin{align}
\tilde{g}_T=\begin{bmatrix} -m_1^{-1}\onesvec_{\mathcal{N}-1}\otimes I_d & m_{2:\mathcal{N}}^{-1}\otimes I_d \end{bmatrix},\label{e:gt}
\end{align}
where $m_{2:\mathcal{N}}^{-1}=\diag(m_2^{-1},\dots,m_{\mathcal{N}}^{-1})$.

The first term on the right hand side of \eqref{e:dyna} is the acceleration due to Coulomb forces and the second term is the acceleration due to thruster actuation. Notice that the Coulomb force is the same if the charges are $q$ or $-q$. So in this paper, the convention is adopted that $q_1 \ge 0$ without having any effect on the attainable Coulomb forces. 

\begin{lemma}\label{l:overact}
Given a desired relative acceleration $u_\text{acc}^\text{des}\in\mathbb{R}^{d(\mathcal{N}-1)}$ and any charge $q^\star\in\mathbb{R}^\mathcal{N}$ and configuration $\tilde{\xi}\in\mathbb{R}^{d(\mathcal{N}-1)}$, there exists a thrust $T^\star\in\mathbb{R}^{d\mathcal{N}}$ such that 
\begin{align}
u_\text{acc}^\text{des}=\tilde{g}_C(\tilde{\xi})\vect(q^\star{q^\star}^\top)+\tilde{g}_TT^\star.\label{e:desacc}
\end{align}
\end{lemma}
\begin{proof}
The matrix $\tilde{g}_T$ in \eqref{e:gt} has more columns than rows and its rows are linearly independent, so it is right invertible, i.e., $\tilde{g}_T{\tilde{g}_T}^\dagger=I_{d(\mathcal{N}-1)}$. Therefore,  $T^\star = {\tilde{g}_T}^\dagger \left(u_\text{acc}^\text{des}-\tilde{g}_C(\tilde{\xi})\vect(q^\star{q^\star}^\top)\right)$ yields  \eqref{e:desacc}. 
\end{proof}

Lemma \ref{l:overact} proves that a HCSF is over-actuated in the sense that any desired acceleration can be generated exclusively using thrusters \cite{Allocation}. Furthermore, Lemma \ref{l:overact} will be useful for showing that the control allocation algorithm proposed later on is always feasible. 

Now define a desired relative configuration $\xi^\text{des}\in\mathbb{R}^{d(\mathcal{N}-1)}$ and define the error between the relative positions and the desired relative configuration:
\begin{align}
\xi = \tilde{\xi}-\xi^\text{des},
\end{align}
and define $\nu=\dot{\xi}$. Using the state vector
\begin{align*}
\Xi = \col(\xi, \nu),
\end{align*}
the dynamics can be written as:
\begin{align}
\dot{\Xi} = f(\Xi)+{g}_C({\xi})\vect\left(qq^\top\right)+{g}_TT,\label{e:dynamix}
\end{align}
where 
\begin{subequations}\label{e:fAB}
\begin{align}
& f(\Xi) = A\Xi, {g}_C({\xi})=B\tilde{g}_C(\xi+\xi^\text{des}), {g}_T=B\tilde{g}_T,\\
& A = \begin{bmatrix} 0_{d(\mathcal{N}-1)\times d(\mathcal{N}-1)} & I_{d(\mathcal{N}-1)}\\ 0_{d(\mathcal{N}-1)\times d(\mathcal{N}-1)}   & 0_{d(\mathcal{N}-1)\times d(\mathcal{N}-1)} \end{bmatrix},\label{eAB1}\\
&B= \begin{bmatrix} 0_{d(\mathcal{N}-1)\times d(\mathcal{N}-1)}  \\  I_{d(\mathcal{N}-1)} \end{bmatrix}.\label{eAB}
\end{align}
\end{subequations}

The objective of this paper is to stabilize the origin of the dynamics \eqref{e:dynamix} using thrust and charge feedback control. Doing so will reconfigure the HCSF to the desired configuration. Since Coulomb actuation requires little propellant, it is desired to minimize the amount of thrust throughout the maneuver. The approach taken in this paper is a control allocation approach with CLFs. 
\subsection{Control Lyapunov Function}
%Using the Lie derivative notation (see \cite[\S 8.1]{Khalil})

Lemma \ref{l:overact} guarantees that any relative acceleration can be produced by a combination of thrusters and Coulomb forces. So the basis of the control allocation algorithm that is proposed later on is a CLF-based acceleration  control law. 

It is assumed that a smooth positive definite function $V(\Xi)$ is known such that, for some $\varepsilon >0$,
\begin{align}
\inf_{u\in\mathbb{R}^{d(\mathcal{N}-1)}}\left\{L_fV+L_{B}Vu+\varepsilon V \right\}\le 0\label{e:CLF}
\end{align}
for all $\Xi\in\mathbb{R}^{2d(\mathcal{N}-1)}$, where $f, A,$ and $B$ are defined in \eqref{e:fAB}.

There are many ways to design $V$ such that \eqref{e:CLF} holds.  A quadratic CLF which satisfies \eqref{e:CLF} can be found by solving the following linear matrix inequalities for $Q\in\mathbb{S}^{2d(\mathcal{N}-1)}$ and $Y\in\mathbb{R}^{d(\mathcal{N}-1)\times2d(\mathcal{N}-1)}$:
\begin{subequations}
\label{e:LMI}
\begin{align}
&AQ+QA^\top +BY+Y^\top B^\top + \varepsilon Q\preceq 0,\\
&Q-\sigma I\succeq 0,
\end{align}
\end{subequations}
where $\sigma>0$. From the solution to \eqref{e:LMI}, the CLF will be 
\begin{align}
V(\Xi) = \Xi^\top P\Xi,\label{quadlyap}
\end{align}
where $P=Q^{-1}$ (see \cite[\S 7]{LMIbook} for more details). 
\subsection{Useful Lemma}
This section is ended with a lemma that will be useful in developing the control allocation algorithm in the following section.

\begin{lemma}\label{l:quad}
Let $a=\begin{bmatrix} a_1 & \dots & a_n\end{bmatrix}$ where $a_i\in\mathbb{R}^{1\times n}$ for each $i=1,\dots,n$, so $a\in\mathbb{R}^{1\times n^2}$. Then $a \vect(xx^\top) = x^\top\left\lfloor a\right\rceil x$ for all $x\in\mathbb{R}^n$, where 
\begin{align*}
\left\lfloor a\right\rceil = \text{He}\left(\begin{bmatrix} a_1 \\ \vdots \\ a_n
\end{bmatrix}\right).
\end{align*}
\end{lemma}
\begin{proof}
This can be proven by expanding 
\begin{align*}
a\vect(xx^\top)&=x_1a_1x+\dots+x_na_nx \\
&= x^\top\begin{bmatrix} a_1 \\ \vdots \\ a_n
\end{bmatrix}x=x^\top\left\lfloor a\right\rceil x.
\end{align*}
\end{proof}

Note that the operator $\left\lfloor a\right\rceil$ defined in Lemma \ref{l:quad} can be coded as the Hermitian of a reshape of $a$ from $1\times n^2$ to $n\times n$.
\section{Control Allocation Algorithm} \label{s:algorithm}
 \subsection{Overview of the Algorithm}
In the proposed control allocation algorithm, a proportion of the decrease of the Lyapunov function will come from the Coulomb actuation and the rest of the Lyapunov function decrease will come from thruster actuation. Let $\eta\in[0,1]$ be the desired proportion of the Lyapunov decrease that comes from Coulomb actuation. In the case $\eta=0$ there is no Coulomb actuation, and in the case $\eta=1$ it is desired that all of the Lyapunov function decrease comes from Coulomb actuation. The proportion $\eta$ is a design parameter. In the simulation section later on, the effect of the choice of $\eta$ will be discussed. 

Let $T^\star$ and $q^\star$ denote the outputs of the control allocation algorithm. The CLF evolves according to:
\begin{align}
\dot{V}=L_fV+L_{g_C}V\vect(qq^\top)+L_{g_T}VT.\label{e:CLFevolve}
\end{align}

If $L_fV+\varepsilon V\le 0$, then no charge or thrust is necessary such that $\dot{V}\le -\varepsilon V$, so $q^\star=0_\mathcal{N}$ and $T^\star=0_{d\mathcal{N}}$.  

If $L_fV+\varepsilon V> 0$, then control is required. First, the charge $q^\star$ will be computed such that 
\begin{align}
L_{g_C}V\vect(q^\star{q^\star}^\top)+\eta(L_fV+\varepsilon V)\le 0\label{e:qstar}
\end{align}
if such a $q^\star$ exists. Otherwise,  $q^\star=0_\mathcal{N}$. Using this $q^\star$, the thrust is computed as follows:
\begin{align}
&T^\star=\arg\min T^\top T\nonumber \\
&\text{s.t. } L_fV+L_{g_C}V\vect(q^\star{q^\star}^\top)+L_{g_T}VT+\varepsilon V\le 0.\label{e:QP}
\end{align}

 Lemma \ref{l:overact} can be used to show that the quadratic program \eqref{e:QP} is always feasible.
 \begin{proposition}
Suppose the CLF $V$ satisfies \eqref{e:CLF}. The quadratic program \eqref{e:QP} is feasible for all $\Xi\in\mathbb{R}^{2d(\mathcal{N}-1)}$ and $q^\star\in\mathbb{R}^\mathcal{N}$.
 \end{proposition}
\begin{proof}
For any given $\Xi\in\mathbb{R}^{2d(\mathcal{N}-1)}$, $\exists u^\star\in \mathbb{R}^{d(\mathcal{N}-1)}$ such that $L_fV+L_{B}Vu^\star+\varepsilon V\le 0$ via  \eqref{e:CLF}. By Lemma  \ref{l:overact}, regardless of the choice of $q^\star$, there exists $T^\star$ such that $u^\star=\tilde{g}_C(\tilde{\xi})\vect(q^\star{q^\star}^\top)+\tilde{g}_TT^\star.$ It follows from the definitions of $g_C$ and $g_T$ in \eqref{e:fAB} that with this $T^\star$ and $q^\star$, $ L_fV+L_{g_C}V\vect(q^\star{q^\star}^\top)+L_{g_T}VT^\star+\varepsilon V\le 0$. The result follows. 
\end{proof}
%Since $V$ satisfies \eqref{e:CLF} and a $T$ can be found to produce any desired acceleration, the quadratic program \eqref{e:QP} is always feasible for any $q^\star$ and $\Xi$. 

Using the chosen $q^\star$ and $T^\star$ yields  an exponential decay of the Lyapunov function, i.e. $\dot{V}\le -\varepsilon V$. 
\subsection{Finding $q^\star$} \label{s:qstar}
The question remains of how to find $q^\star$ that satisfies \eqref{e:qstar}. Using Lemma \ref{l:quad}, the left-hand side of \eqref{e:qstar} can be rewritten as quadratic in $q^\star$:
\begin{align}
{q^\star}^\top\left\lfloor L_{g_C}V\right\rceil q^\star+\eta(L_fV+\varepsilon V).\label{e:CLFevolve2}
\end{align}

Consider the following eigenvalue decomposition:
\begin{align}
\left\lfloor L_{g_C}V\right\rceil = R\Lambda R^\top\label{e:eigendecomp1}
\end{align}
where $R$ is a unitary matrix, $\Lambda=\text{diag}(\Lambda_{11},\dots,\Lambda_{\mathcal{N}\mathcal{N}})$ and $\Lambda_{11}\le \dots\le \Lambda_{\mathcal{N}\mathcal{N}}$. By defining $\tilde{q}^\star=R^\top q^\star$, \eqref{e:CLFevolve2} is equivalent to 
\begin{align}
\sum_{i=1}^\mathcal{N}\Lambda_{ii} {\tilde{q}_i}^{\star^2}+\eta(L_fV+\varepsilon V). \label{e:decomprewrite}
\end{align}
From \eqref{e:decomprewrite}, it is clear that if $\Lambda_{11}<0$, then using 
\begin{subequations}\label{e:qstarbest}
\begin{align}
&\tilde{q}_1^\star=\sqrt{\dfrac{-\eta(L_fV+\varepsilon V)}{\Lambda_{11}}},  \label{e:qstar1}\\
&\tilde{q}_i^\star=0, i=2,\dots,\mathcal{N}, \label{e:qstarzeros}
\end{align}
\end{subequations}
yields 
\begin{align}
\sum_{i=1}^\mathcal{N}\Lambda_{ii} {\tilde{q}_i}^{\star^2}+\eta(L_fV+\varepsilon V)\le 0. 
\end{align}
Therefore, using 
\begin{align}
q^\star=R\tilde{q}^\star,\label{e:transform}
\end{align}
 where $\tilde{q}^\star$ comes from \eqref{e:qstarbest} satisfies \eqref{e:qstar}.

In the Coulomb dynamics \eqref{e:truedyn}, the right hand side only contains charge products $q_iq_j$ where $i\ne j$. Therefore, the diagonal elements of the matrix $\left\lfloor L_{g_C}V\right\rceil$ are all equal to zero. This means that 
\begin{align*}
\text{Tr}\left(\left\lfloor L_{g_C}V\right\rceil\right)=\sum_{i=1}^\mathcal{N}\Lambda_{ii}=0.
\end{align*}
It follows, therefore, that, unless $\Lambda_{ii}=0$ for all $i=1,\dots, \mathcal{N}$, $\Lambda_{11}<0$. Since $\left\lfloor L_{g_C}V\right\rceil$ is symmetric,  its eigenvalues are all equal to zero if and only if it is the zero matrix. 

As $\Xi\to 0$,   $\left\lfloor L_{g_C}V\right\rceil\to 0$, so $\Lambda_{11}$, though negative, will have smaller and smaller magnitude. So using \eqref{e:qstar1} could yield impractically large amounts of charge. One way to guard against this is to specify a maximum charge $q_{\max}>0$ and substitute \eqref{e:qstar1} for the following:
\begin{align}
\tilde{q}_1^\star=\min\left\{q_{\max},\sqrt{\dfrac{-\eta(L_fV+\varepsilon V)}{\Lambda_{11}}}\right\}. \label{e:qstar1min}
\end{align}
Using $q^*$ computed by \eqref{e:transform} where $\tilde{q}^\star$ is from \eqref{e:qstar1min} and \eqref{e:qstarzeros}  always produces a decrease in the CLF. Moreover, it produces the closest $L_{g_C}V\vect(q^\star{q^\star}^\top)$ to $-\eta(L_fV+\varepsilon V)$ where $\|q^\star\|\le q_{\max}$.

\subsection{Summary and Implementation}\label{s:summary}
The process for choosing charges and thrusts described above is summarized as Algorithm 1\footnote{Note that EigR$\left(\left\lfloor L_{g_C}V\right\rceil\right)$ in Algorithm 1 denotes a function that  returns the diagonal and unitary matrices from \eqref{e:eigendecomp1}.
}.

%According to   for any $\Delta >0$ and $\Delta_\Xi$ such that $\|\Xi(0)\|\le \Delta_\Xi$, there exists a sampling time $\Delta t^\star$ such that if $\Delta t\in (0, \Delta t^\star]$ then 
%\begin{align}
%\|\Xi(t)\|\le \beta(\|\Xi(0)\|,t)+\Delta, \label{e:stab}
%\end{align}
%where $\beta\in\mathcal{KL}$\footnote{The definition of a $\mathcal{KL}$ function can be found in \cite[Definition 4.1]{Khalil}. The condition \eqref{e:stab} means that as $t\to\infty$, $\Xi(t)$ converges to an invariant ball of radius $\Delta$ in finite time.}. In other words, it is expected that there will be some steady-state error and that error can be made arbitrarily small by choosing $\Delta t$ small enough \cite[\S 10.5]{HybridControl}. If $\Delta t$ is too small, however, Algorithm 1 cannot be computed sufficiently fast to implement the control.  
\begin{algorithm}
\begin{algorithmic}
\caption{CLF-based Control Allocation}
\REQUIRE $\eta\in[0,1], \varepsilon>0, q_{\max}>0, \Xi\in \mathbb{R}^{2d(\mathcal{N}-1)}$
\STATE Compute $V, L_fV,\left\lfloor L_{g_C}V\right\rceil, L_{g_T}V$ 
\IF{$L_fV+\varepsilon V\le 0$}
\STATE $q^\star \leftarrow 0_{\mathcal{N}}$
\STATE $T^\star \leftarrow 0_{d\mathcal{N}}$
\ELSE
\STATE $\Lambda, R\leftarrow$EigR$\left(\left\lfloor L_{g_C}V\right\rceil\right)$
\IF{$\Lambda_{11}\ge 0$}
\STATE $\tilde{q}_1^\star\leftarrow 0$
\ELSE
\STATE $\tilde{q}_1^\star\leftarrow \min\left\{q_{\max},\sqrt{\dfrac{-\eta(L_fV+\varepsilon V)}{\Lambda_{11}}} \right\}$
\ENDIF
\STATE $q^\star\leftarrow R\begin{bmatrix}\tilde{q}_1^\star \\ 0_{\mathcal{N}-1}\end{bmatrix}$
\STATE $T^\star\leftarrow \eqref{e:QP}$
\ENDIF
\RETURN $q^\star,T^\star$
\end{algorithmic}
\end{algorithm}

If Algorithm 1 is implemented in a continuous-time fashion, the result would be the exponential stabilization of the origin of \eqref{e:dynamix}; however, Algorithm 1 cannot be implemented in a continuous-time fashion.  The charges and thrusts computed by Algorithm 1 are, instead, implemented in a sample-and-hold fashion with a sampling period $\Delta t >0$. Implementing a continuos-time controller via sample-and-hold is referred to as emulating the continuous-time controller   \cite[\S 2]{NecsicSurvey01}.

When emulating an asymptotically stabilizing continuous-time controller, the result is expected to be semiglobal practical stabilization of the origin \cite{NecsicSurvey01,4177901,7525591}. Roughly speaking, this means that there will be some steady-state error and that error can be made arbitrarily small by choosing $\Delta t$ small enough. The results in \cite{NecsicSurvey01,4177901,7525591} for semiglobal practical stabilization assume that the feedback law is continuous. Let $(q^\star,T^\star) = K_{\varepsilon,\eta}(\Xi)$ denote the function which provides the output of Algorithm 1 for fixed $\varepsilon >0$ and $\eta\in[0,1]$. This function $K_{\varepsilon,\eta}$ is not a continuous function due to the use of the eigendecomposition. Even though $\left\lfloor L_{g_C}V\right\rceil$ and its eigenvalues vary continuously with $\Xi$, its eigenvectors do not  in general \cite[\S 2.5.3]{PerturbationTheory}. In simulations, it is observed that using Algorithm 1 with $\Delta t$ small enough is stabilizing with steady-state error. Due to lack of continuity, it may not be the case that the error can be made arbitrarily small by choosing the sampling time appropriately. 

\section{Relationship between Algorithm 1 and Pointwise Optimal Allocation Approach}\label{s:qcqprel}
Consider the following QCQP which represents the pointwise optimization approach to control allocation with CLFs:
\begin{align}
&\min \gamma_1 q^\top q+\gamma_2 T^\top T\nonumber \\
&\text{s.t. } L_fV+{q}^\top\left\lfloor L_{g_C}V\right\rceil q+L_{g_T}VT+\varepsilon V\le 0.\label{e:QCQP}
\end{align}
%The constants $\gamma_1,\gamma_2\ge 0$ are related to the tradeoff between thrust and charge input. 

Unless $L_fV+\varepsilon V<0$, the optimal $q^\star$ and $T^\star$ to \eqref{e:QCQP} will satisfy the constraint with equality \cite[Appendix B]{park2017generalheuristicsnonconvexquadratically}, that is:
\begin{align}
 L_fV+{q^\star}^\top\left\lfloor L_{g_C}V\right\rceil q^\star+L_{g_T}VT^\star+\varepsilon V= 0.\label{e:equalityconstraint}
\end{align}
The equality constraint \eqref{e:equalityconstraint} implies that a proportion of the Lyapunov decrease will occur due to the Coulomb actuation and the rest occurs due to the thruster actuation. The amount of decrease each is responsible for is mediated by $\gamma_1$ and $\gamma_2$.

It is easy to see that the optimal value of  \eqref{e:QCQP} with $\gamma_1 =0$ and $\gamma_2 > 0$ is the same as the output of Algorithm 1 with $\eta =1$ (without the maximum charge constraint $q_{\max}$). 
\subsection{The Parameter $\eta$}
Now, the parameter $\eta$ from Algorithm 1 is analogized to the parameters in \eqref{e:QCQP}. The rest of this section follows reasoning similar to \cite[Appendix B]{park2017generalheuristicsnonconvexquadratically}. It is assumed throughout the rest of this section that $L_fV+\varepsilon V> 0$. Lastly, for simplicity, $\gamma_1=1$ and $\gamma_2>0$ for the rest of this section. %Increasing $\gamma_2$ increases the penalty on thrust input which leads to lower thrust. 

Consider the eigenvalue decomposition \eqref{e:eigendecomp1}, and the change of variables $\tilde{q}=R^\top q$ and $\tilde{T}=T$, so  \eqref{e:QCQP} is equivalent to the following:
\begin{align}
&\min \tilde{q}^\top \tilde{q}+\gamma_2 \tilde{T}^\top \tilde{T}\nonumber \\
&\text{s.t. } L_fV+{\tilde{q}}^\top\Lambda  \tilde{q}+L_{g_T}V\tilde{T}+\varepsilon V\le 0.\label{e:QCQPcoord}
\end{align}

The Lagrangian of \eqref{e:QCQPcoord} is the following:
\begin{align}
\mathcal{L} = \begin{bmatrix} \tilde{q} \\ \tilde{T} \end{bmatrix}^\top\begin{bmatrix} I + \mu \Lambda & 0 \\ 0 & \gamma_2 I \end{bmatrix}\begin{bmatrix} \tilde{q} \\ \tilde{T} \end{bmatrix}+\mu( L_{g_T}V\tilde{T}+L_fV+\varepsilon V),\nonumber
\end{align}
where $\mu\in\mathbb{R}$ is the Lagrange multiplier. The QCQP \eqref{e:QCQPcoord} is always feasible due to the analysis presented in \S\ref{s:algorithm}, so there must exist $\mu$ such that $I+\mu\Lambda\succeq 0$ where minimizing $\mathcal{L}$ also satisfies the constraint.

First consider the optimality conditions for the Lagrangian:
\begin{align}
&\frac{\partial \mathcal{L}}{\partial \tilde{q}} = 2\tilde{q}^\top (I+\mu \Lambda)=0, \nonumber\\
&\implies (I+\mu \Lambda)\tilde{q}^\star=0,\label{e:optimalq}
\end{align}
and 
\begin{align}
&\frac{\partial \mathcal{L}}{\partial \tilde{T}} = 2\gamma_2 \tilde{T}^\top +\mu L_{g_T}V =0,\nonumber\\
&\implies \tilde{T}^\star = -\frac{\mu}{2\gamma_2}L_{g_T}V^\top. \nonumber
\end{align}
Moreover, due to complementary slackness, the constraint will be satisfied in equality:
\begin{align}
L_fV+{\tilde{q}}^{\star^\top}\Lambda  \tilde{q}^\star+L_{g_T}V\tilde{T}^\star+\varepsilon V=0.\label{equality}
\end{align}

Now consider \eqref{e:optimalq} and two cases: if $\mu$ is such that $I+\mu\Lambda$ is nonsingular, then $\tilde{q}^\star=0$, which is not a desirable solution. So consider the case where $I+\mu\Lambda$ is singular. Since, by the argumentation in \S \ref{s:qstar}, $\Lambda_{11}<0$ unless $\left\lfloor L_{g_C}V\right\rceil=0$ (in which case, the original problem reduces to a QP), the only way for $I+\mu \Lambda$ to be positive semidefinite and singular is if 
\begin{align*}
\mu = -\frac{1}{\Lambda_{11}}.
\end{align*}

With this $\mu$, it follows from \eqref{e:optimalq} that $\tilde{q}_i^\star=0$ for $i=2,\dots,\mathcal{N}$. Putting this all together with \eqref{equality} yields:
\begin{align}
\Lambda_{11}\tilde{q}_1^{\star^2}+\frac{1}{2\gamma_2\Lambda_{11}}L_{g_T}VL_{g_T}V^\top+L_fV+\varepsilon V=0.\label{alltogether}
\end{align} 

Notice that if $\gamma_2$ is too small, the equation \eqref{alltogether} cannot hold. Hence, to guarantee that an optimal solution with nonzero $\tilde{q}^\star$ occurs, $\gamma_2$ needs to be updated at every sampling instance based on the values of the other terms in \eqref{alltogether}. In essence, Algorithm 1 returns the optimal arguments of \eqref{e:QCQP} where $\gamma_1=1$ and 
\begin{align*}
\gamma_2 = -\frac{L_{g_T}VL_{g_T}V^\top}{2(1-\eta)\Lambda_{11}(L_fV+\varepsilon V)}.
\end{align*}
Since $\eta$ belongs to a bounded interval $[0,1]$ and it is intuitive enough to work with, it is perhaps a better tradeoff parameter to use than $\gamma_2$ to achieve similar results. Algorithm 1 can be thought of as a good heuristic for the pointwise optimal allocation approach.

\section{Numerical Example}\label{s:example}
This example is of a formation of $\mathcal{N}=4$ hybrid Coulomb spacecraft. The masses of the spacecraft are $m_1=100$ kg, $m_2=96$ kg, $m_3=130$ kg, and $m_4=100$ kg. The spacecraft initial positions are the following:
\begin{align*}
x_1 = \begin{bmatrix}-100\\ -90\end{bmatrix},
x_2 = \begin{bmatrix}-50\\ 30\end{bmatrix},
x_3 = \begin{bmatrix}100\\ 60\end{bmatrix},
x_4 = \begin{bmatrix} 100\\ -60\end{bmatrix}.
\end{align*}
The spacecraft start from rest, so all the velocities are zero. 

The desired relative formation is the following:
\begin{align*}
\xi^{des} = \begin{bmatrix}0&150&150&150&150&0\end{bmatrix}^\top,
\end{align*}
which describes a square with sides of 150 m length. This makes the initial deviation of the relative position from the desired relative position 
\begin{align*}
\xi(0) = \begin{bmatrix} 50& -30 & 50& 0&50&30\end{bmatrix}^\top.
\end{align*}

The quadratic CLF \eqref{quadlyap} where
\begin{align*}
P = \begin{bmatrix}0.995057 I_6 &  0.00497061I_6\\  0.00497061I_6 & 0.995057 I_6 \end{bmatrix}
\end{align*}
was found by solving the LMI in \eqref{e:LMI} where $\varepsilon =0.01$ and $\sigma=1$ using \verb|Convex.jl| \cite{convexjl} and the \verb|SCS| solver \cite{ocpb:16}.

All of the simulations in this section were performed in \verb|Julia| using  \verb|RobotDynamics.jl|\footnote{https://rexlab.ri.cmu.edu/RobotDynamics.jl/stable/}. More details about the implementation of Algorithm 1 are discussed in \S \ref{s:compburd}

 With $\eta=0.99$ and $\Delta t=0.1$ secs, the trajectory of $\Xi$ is plotted in Fig. \ref{f:traj} where it can be seen that $\Xi$ settles towards 0 with a very small steady state error (within cms of the desired configuration, which is very good considering $\|\xi(0)\|=96.4$ m).

\begin{figure}[h]
\centerline{\includegraphics[width=0.99\columnwidth]{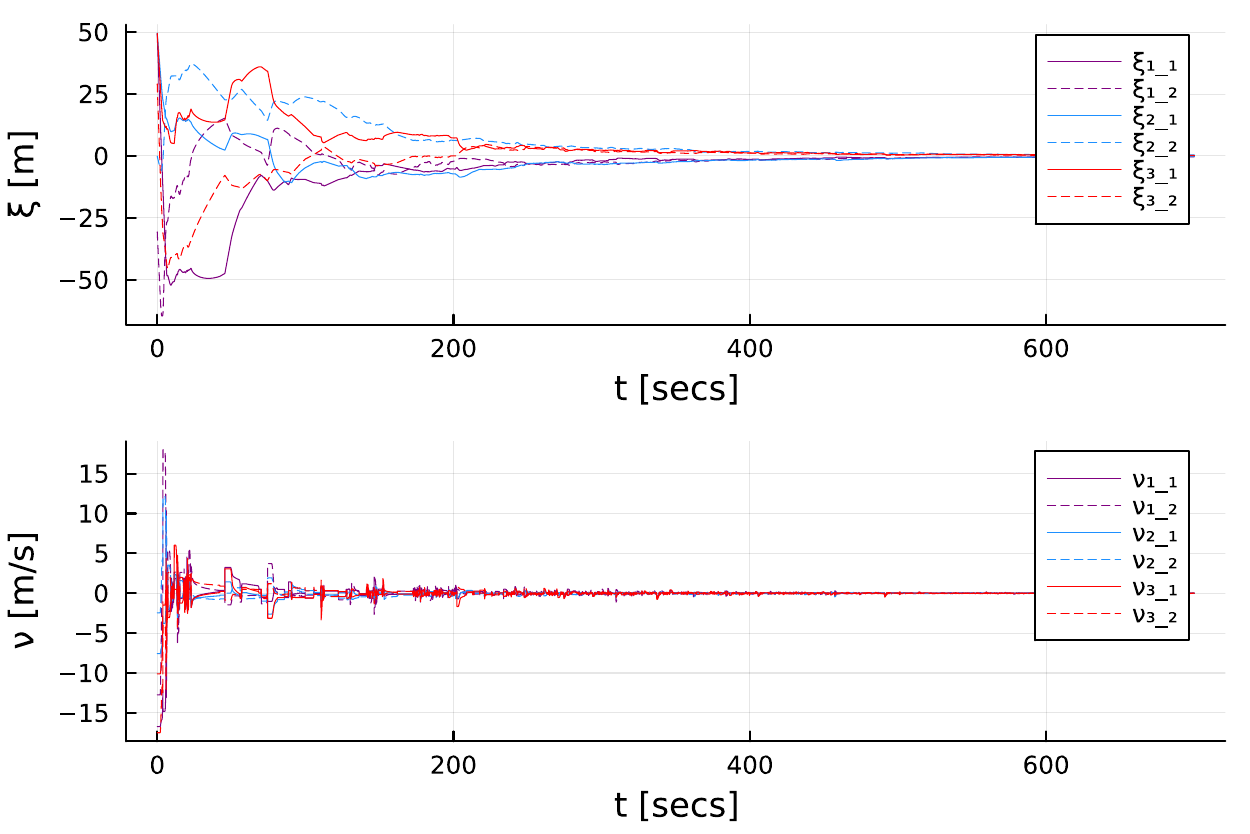}}
\caption{Trajectory when $\eta =0.99$.}
\label{f:traj}
\end{figure}

The sequence of charges are plotted in Fig \ref{f:charges} and the thrust magnitude (i.e. $\|T\|$) is plotted in Fig. \ref{f:thrust}. The largest inputs occur in the beginning which constitute the initial jerk to get the spacecraft moving towards the desired configuration.
\begin{figure}[h]
\centerline{\includegraphics[width=0.99\columnwidth]{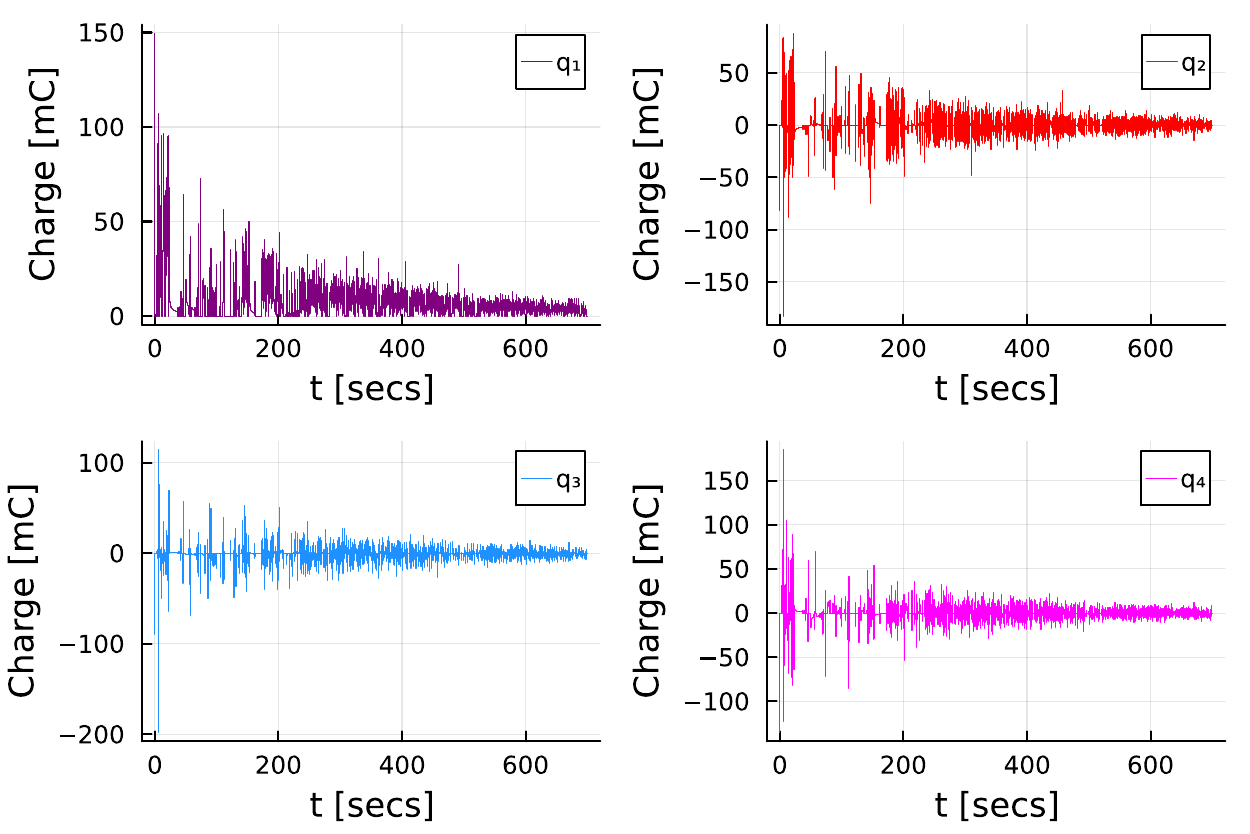}}
\caption{Charges  when $\eta =0.99$.}
\label{f:charges}
\end{figure}

\begin{figure}[h]
\centerline{\includegraphics[width=0.99\columnwidth]{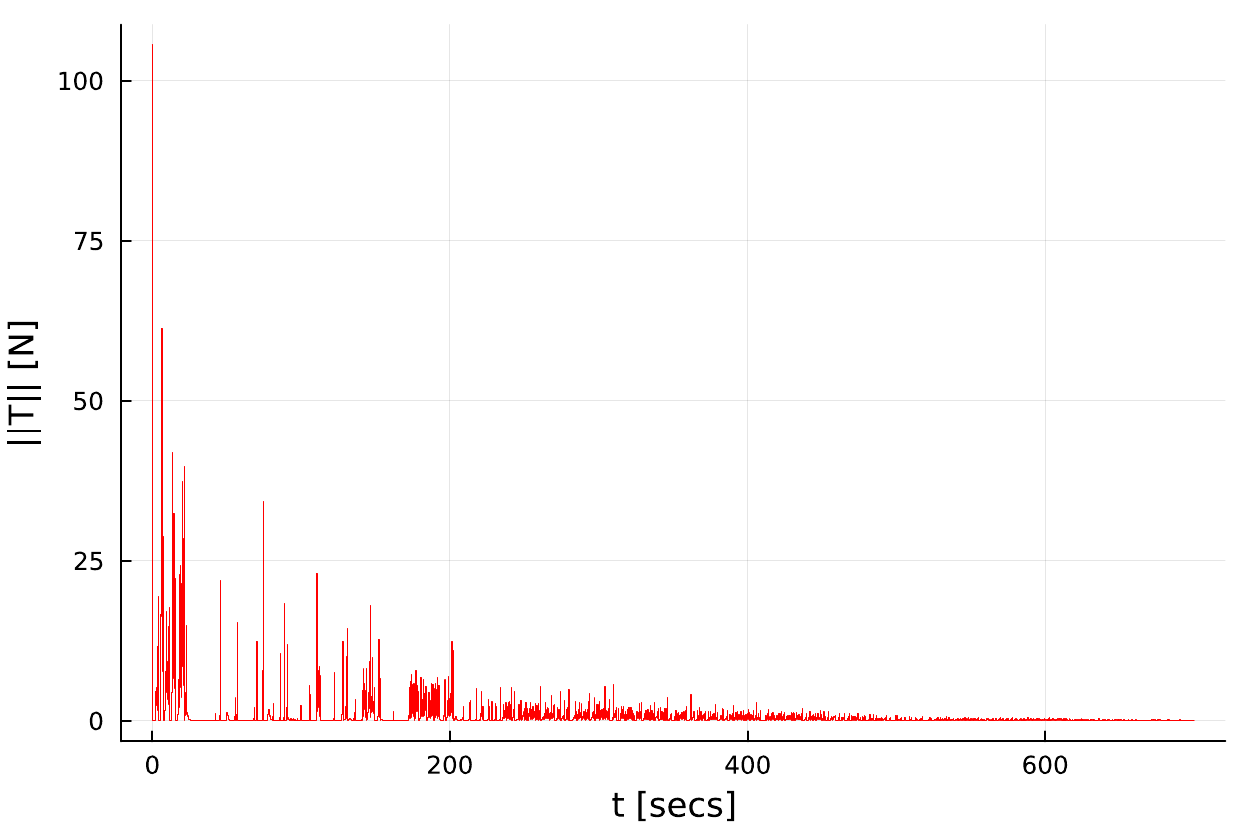}}
\caption{Thrust magnitude  when $\eta =0.99$.}
\label{f:thrust}
\end{figure}

It is difficult to see some of the detail of the charge and thrust sequences in Fig. \ref{f:charges} and Fig. \ref{f:thrust}, so  Fig. \ref{f:zoom1} and Fig. \ref{f:zoom2} show a zoomed in portion of the inputs between 15 secs and 25 secs to get a clearer image of what the inputs look like. Due to the nonlinear and high-dimensional nature of the charge inputs, it is difficult to intuit the behavior of the charges. Some behavior resembling the alternating `push-pull' behavior via alternating charge signs described in \cite{tahir2024modelpredictivecontrolcollinear} can be seen. Notice also that there are portions (e.g. around the 16 secs mark) where thrust and charges are equal to zero. Since the Lyapunov function is decreasing sufficiently in open-loop at that time (i.e. $L_fV+\varepsilon V\le 0$), no input is required.  

\begin{figure}[h]
\centerline{\includegraphics[width=0.99\columnwidth]{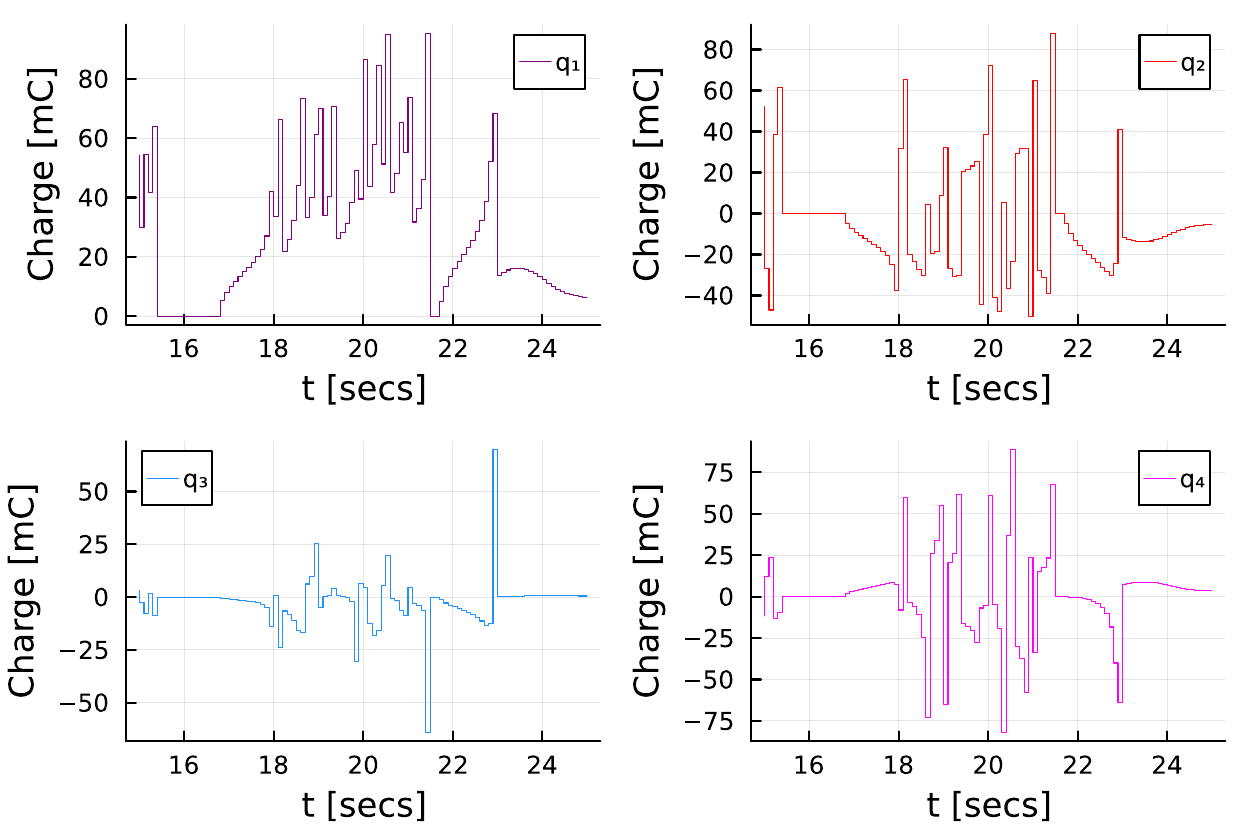}}
\caption{Zoom in on a portion of the charges when $\eta=0.99$.}
\label{f:zoom1}
\end{figure}

\begin{figure}[h]
\centerline{\includegraphics[width=0.99\columnwidth]{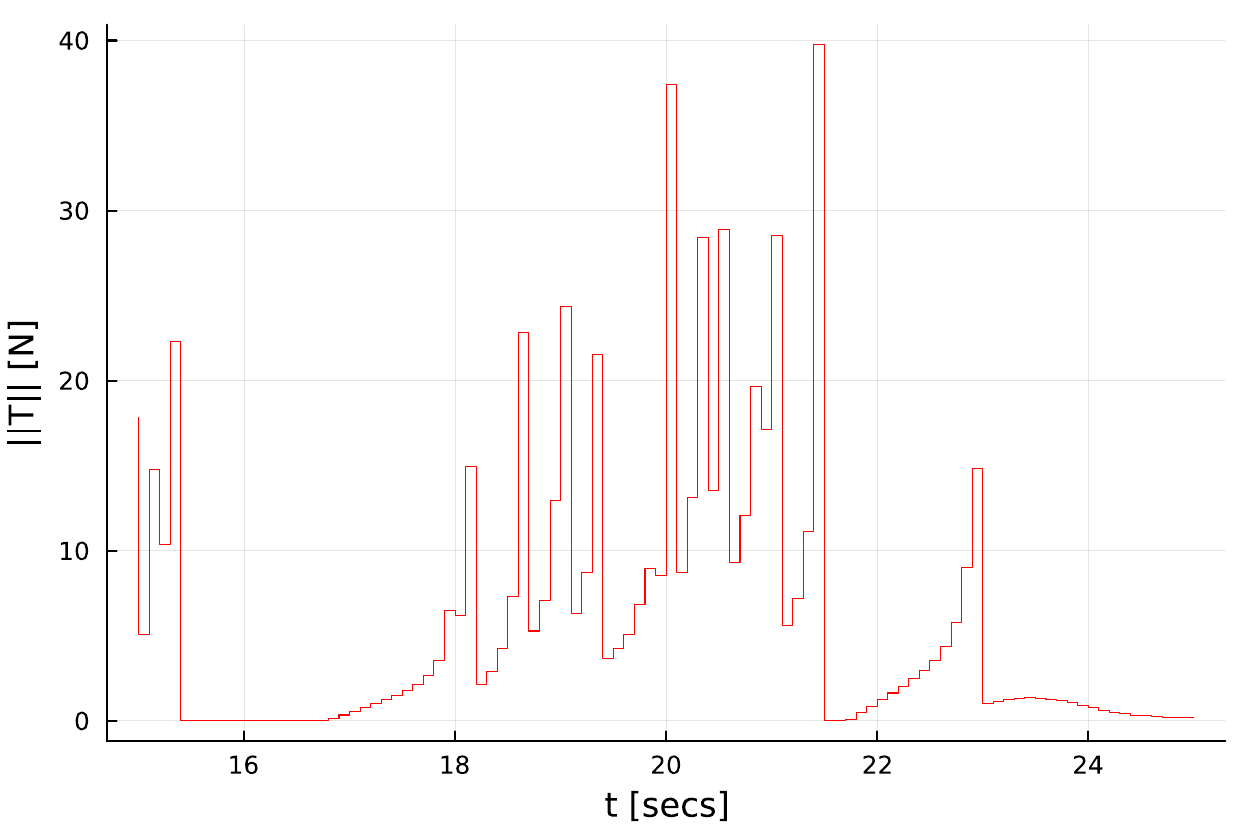}}
\caption{Zoom in on a portion of the thrust magnitude when $\eta=0.99$.}
\label{f:zoom2}
\end{figure}

\subsection{Computational Burden}\label{s:compburd}
The simulation and control allocation algorithm was performed in \verb|Julia|. The quadratic program \eqref{e:QP} was solved using OSQP \cite{osqp}.  

According to \verb|BenchmarkTools.jl|\footnote{https://juliaci.github.io/BenchmarkTools.jl/stable/}, for $\mathcal{N}=4$, Algorithm 1 takes  an average of 1.158 msecs to compute from a trial of 4243 samples (standard deviation 5.396 msecs) on a Macbook Pro laptop with an Intel i5 processor at 2.4 GHz. 

\subsection{Effect of $\eta$ on thrust}
The impulse delivered by the thrust over the length of the maneuver is defined as the following:
\begin{align*}
I_t = \int_0^{t_f}\|T\|dt.
\end{align*}
A lower value of $I_t$ implies less propellant is used in the maneuver. For different values of $\eta\in[0,1]$, the thrust impulse is computed where $t_f=700$ secs, and the results are tabulated in Table \ref{t:impulse}.

\begin{table}[h!]

\caption{Impulse for different values of $\eta$} 
\begin{center}
\begin{tabular}{ |c|c|c| }
\hline
$\eta$ & $I_t$ (kN$\cdot$s)\\
\hline
0.0 & 2.908  \\ 
0.01 & 4.142 \\
0.1 & 10.025 \\
0.25 & 15.98 \\
0.5 & 14.940 \\
0.8 & 6.893 \\
0.9 & 4.625 \\ 
0.96 & 1.867 \\
0.97 & 1.609 \\
0.98 & 0.7664 \\
0.99 & 0.4903 \\
1.0 & $0$\\
\hline
\end{tabular}
\label{t:impulse}
\end{center}
\end{table}

The case where $\eta = 0.0$ is the case where only thruster forces are used. It is interesting to see that it is not the case that adding in Coulomb actuation (i.e. $\eta > 0.0$) necessarily implies a reduction in the propellant usage. In fact a decrease is not seen until $\eta > 0.95$. This implies that the best role of the thrusters is to provide a small amount of compensation for the underactuatedness of the Coulomb forces. For the case of $\eta=0.99$, there is an 83.1\% reduction in propellant compared to using solely thrusters.

No thruster actuation occurs when $\eta=1.0$; however, in this case, the steady state error is noticeably large when $\Delta t=0.1$ secs. See Fig. \ref{f:etaone}. The other values of $\eta$ in Table \ref{t:impulse} have similar steady-state performance to the case  $\eta=0.99$  shown in Fig. \ref{f:traj}. Decreasing $\Delta t$ to 0.01 secs did not have an appreciable effect on the steady-state error. This is possibly a consequence of the lack of continuity of the control law being emulated (see the discussion in \S \ref{s:summary}). Decreasing $\Delta t$ further would be impractical because of the computational burden of Algorithm 1. %From experimentation, it seems as though $\Delta t$ needs to be much smaller than $\Delta t=0.01$ secs when $\eta =1.0$ to get similar steady-state error as the case of $\eta =0.99$. It is difficult to know how small $\Delta t$ needs to be to guarantee that steady-state performance level--it could possibly be on the order of $\mu$secs or less.  Nevertheless, such small sampling times cannot be implemented because of the computational burden of Algorithm 1. %Moreover, recall from \S \ref{s:summary} that there exists a . 

\begin{figure}[h]
\centerline{\includegraphics[width=0.99\columnwidth]{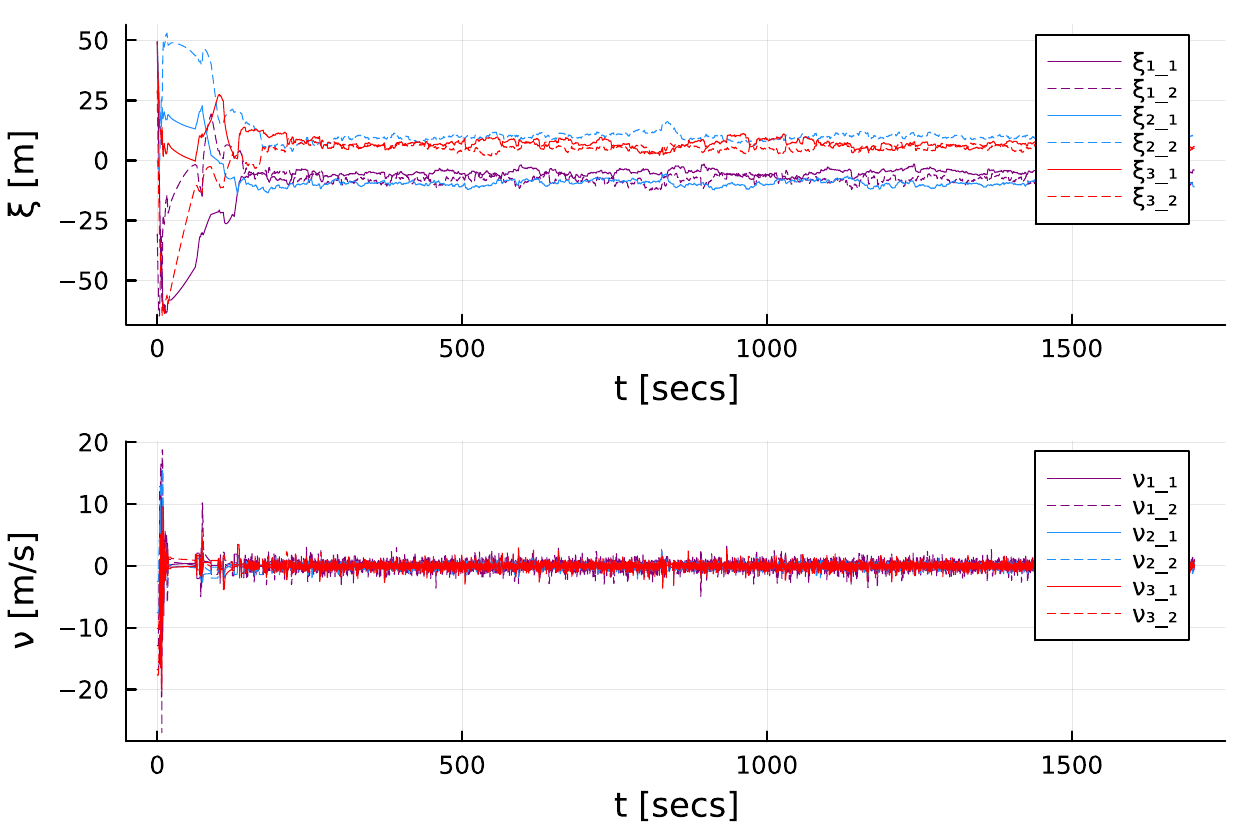}}
\caption{Trajectory when $\eta=1.0$.}
\label{f:etaone}
\end{figure}

The parameter $\eta$ can be varied throughout the maneuver. Therefore, a strategy for further reducing propellant could be to use $\eta=1.0$ until the formation cannot progress further. At that point, a switch to $\eta =0.99$ would occur. In Fig. \ref{f:etasche}, this strategy is implemented where the switch occurs at the 300 secs mark. Using this strategy yields a much better steady-state error performance than  Fig. \ref{f:etaone} and the impulse required is $I_t=0.421164$ kN$\cdot$s, which is smaller than the case where $\eta=0.99$ constantly over the entire maneuver. This is an 85.5\% reduction in propellant compared to solely using thrusters.

\begin{figure}[h]
\centerline{\includegraphics[width=0.99\columnwidth]{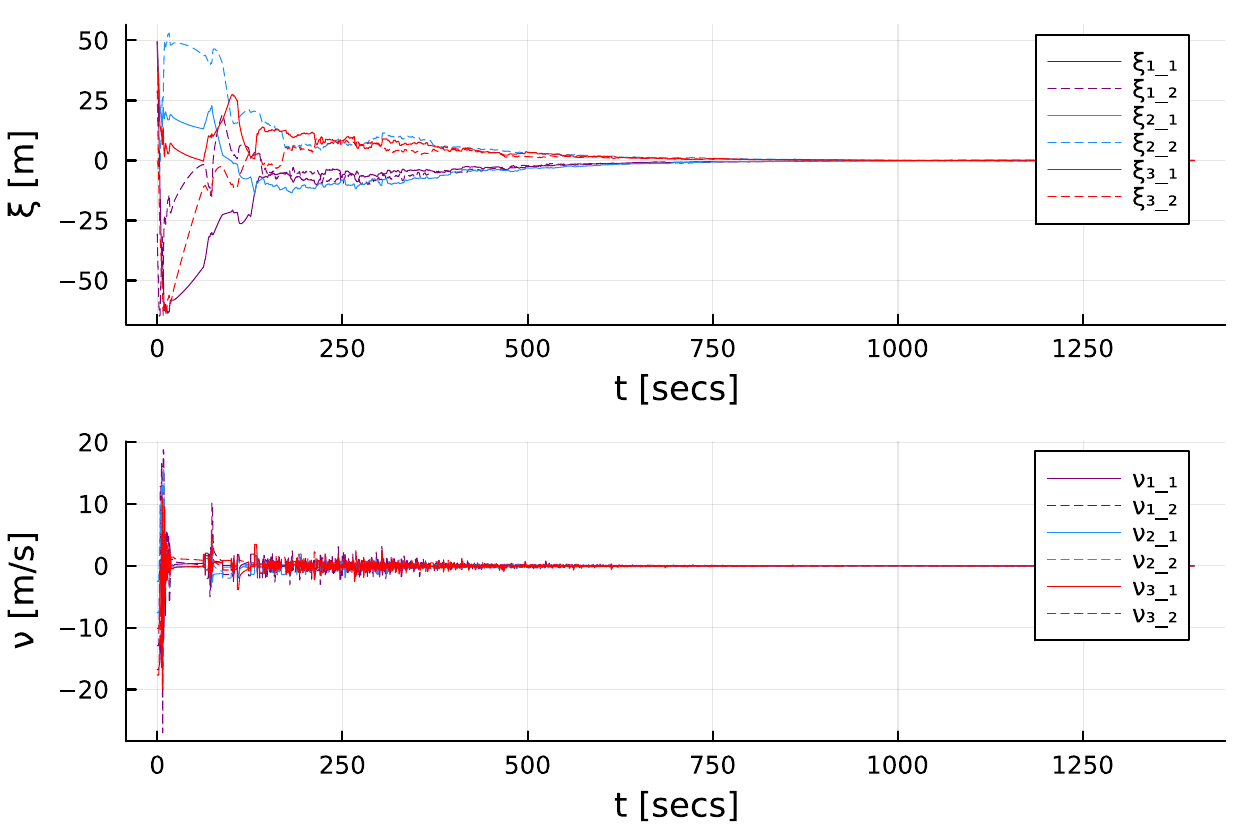}}
\caption{Trajectory when $\eta =1.0$ for the first 300 secs and then $\eta=0.99$ for the remainder of the time. }
\label{f:etasche}
\end{figure}

The simulations discussed above show that $\eta$ is an important design parameter. They show that using a small amount of conventional thrusting is useful which could be because it provides forces that cannot be attained by Coulomb actuation in small corrections.

\section{Conclusions}\label{s:conc}
It was shown in this paper that a HCSF is over-actuated in the sense that any acceleration can be produced by a combination of thrusters and Coulomb forces. Therefore, an acceleration-based CLF is used. Algorithm 1 finds charges and thrusts such that the CLF decreases. A tradeoff parameter $\eta$ was used to designate how much of the CLF decrease should come from the Coulomb forces. The thrusters provide the rest of the CLF decrease. In simulation, by carefully choosing $\eta$, an 85\% reduction in propellant was observed compared to using solely thrusters. 

\bibliography{MLbib}

\end{document}